\documentclass[11pt, reqno]{amsart}
\usepackage{amsfonts,amssymb,amscd,amsmath,mathrsfs,amsthm,
hyperref}
\usepackage{mathtools}
\mathtoolsset{showonlyrefs}

\newcommand{\C}{\mathbb{C}}
\newcommand{\R}{\mathbb{R}}

\newcommand{\Z}{\mathbb{Z}}
\newcommand{\mc}[1]{\mathcal{#1}}

\renewcommand{\H}{\mathfrak{H}}

\renewcommand{\L}{\mathfrak{L}}

\newtheorem{theorem}{Theorem}
\newtheorem{lemma}[theorem]{Lemma}
\newtheorem*{definition}{Definition}
\newtheorem{proposition}[theorem]{Proposition}
\newtheorem{corollary}[theorem]{Corollary}

\author{Felipe Gon\c{c}alves}
\address{Hausdorff Center for Mathematics, Universit\"at  Bonn, Endenicher Allee 60, 53115  Bonn, Germany}
\email{goncalve@math.uni-bonn.de}

\author{Friedrich Littmann}
\address{Department of Mathematics \# 2750, North Dakota State University, Fargo, ND 58108-6050}
\email{Friedrich.Littmann@ndsu.edu}

\title[Weighted mean convergence]{Mean convergence of entire interpolations in weighted space}

\keywords{Lagrange interpolation, Hermite interpolation, model space, de Branges space, exponential type, Marcinkiewicz Zygmund inequality. \\ \indent Data sharing not applicable to this article as no datasets were generated or analysed during the current study}
 
\begin{document}
 
\begin{abstract} We investigate the convergence of entire Lagrange interpolations and of Hermite interpolations of exponential type $\tau$, as $\tau\to \infty$, in weighted $L^p$-spaces on the real line. The weights are reciprocals of entire functions that depend on $\tau$ and may be viewed as smoothed versions of a target weight $w$. The convergence statements are obtained from weighted Marcinkiewicz inequalities for entire functions. We apply our main results to deal with power weights.
\end{abstract}

 \maketitle

 \section{Introduction}

Let $w:\R \to \R$ be a suitable weight function such that the space of entire functions $F:\C\to\C$ of exponential type $\tau$ with $Fw \in L^p(\R)$ is a Banach space for any $1<p<\infty$, which we denote by $\mc{B}^p(\tau,w)$. This article discusses under which conditions on $f$ and $w$ convergence of entire interpolants of $f$ in the weighted spaces $\mc{B}^p(\tau,w)$ takes place as $\tau\to\infty$. For $\tau>0$ we seek a discrete set $\Lambda_\tau\subseteq \R$ such that:
 \begin{itemize}
\item (Mean convergence of Lagrange interpolation) There exists
$\L_\tau f \in \mc{B}^p(\tau,w)$ with 
\[
\L_\tau f(\lambda) = f(\lambda)
\]
for all $\lambda\in \Lambda_\tau$, and
\begin{align}\label{convergence-condition-one}
\lim_{\tau\to \infty} \|(f - \L_\tau f)w\|_p =0.
\end{align}

\item (Mean convergence of Hermite interpolation) There exists $\H_\tau f\in \mc{B}^p(2\tau,w^2)$ with 
\begin{align*}
\H_\tau f(\lambda) &= f(\lambda),\\
\H_\tau' f(\lambda) &= f'(\lambda)
\end{align*}
for all $\lambda\in \Lambda_\tau$, and 
\begin{align}\label{convergence-condition-two}
\lim_{\tau\to \infty} \|(f - \H_\tau f)w^2\|_p =0.
\end{align}
\end{itemize}

The precise definitions of $\L_\tau f$ and $\H_\tau f$ are given in \eqref{Ltau-def} and \eqref{Htau-def} below. We show in Theorems \ref{intro-thm1} and \ref{intro-thm2} that certain entire functions have the property that their zero sets provide interpolation nodes $\Lambda_\tau$ with the desired properties. The restriction $p\notin \{1,\infty\}$ is inherent in the problem; the interpolations constructed below are not necessarily in $L^1(w)$ and may be unbounded for fixed $x$ as a function of $\tau$.

It is clear that the condition $fw\in L^p(\R)$ is not strong enough for statements about interpolation, and it turns out that continuity of $fw$ is stronger than necessary.  We define  the collection $\mc{R}_p(w)$ of functions $f$ such that $fw$ is Riemann integrable and in $L^p(\R)$, and we define $\mc{R}^{(1)}_p(w)$ to be the collection of $f$ such that $f$ is absolutely continuous, $fw, f'w\in L^p(\R)$, and $f'w$ is Riemann integrable.

 It is well known that a main ingredient of the convergence statements \eqref{convergence-condition-one} and \eqref{convergence-condition-two}   is a lower Marcinkiewicz-Zygmund inequality 
\begin{align}\label{lower-MZ}
\int_\R |F(x)w(x)|^p dx \le \frac{C_p}{\tau} \sum_{\lambda\in \Lambda_\tau} |F(\lambda) w(\lambda)|^p
\end{align}
with $C_p$ independent of $\tau$, valid for all functions $F\in \mc{B}^p(\tau,w)$.  We now briefly describe the general strategy of obtaining \eqref{convergence-condition-one} from \eqref{lower-MZ} under the assumption that $\cup _{\tau>0} \mc{B}^p(\tau,w)$ is dense in $L^p(w)$. Let $F_\tau\in \mc{B}^p(\tau,w)$ (not necessarily of interpolating nature) with $\|(f - F_\tau f)w\|_p\to 0$ as $\tau\to \infty$. If $\sigma\ge \tau> 0$ are given and the interpolation $\L_\sigma$ satisfies $\L_\sigma F = F$ for all $F\in \mc{B}^p(\sigma,w)$, then $F_\tau=\L_\sigma F_\tau$ and we obtain
\begin{align}\label{interpolation-difference-rep}
f - \L_\sigma f =  f - F_\tau + \L_\sigma(F_\tau -  f).
\end{align}
If the difference of consecutive elements in $\Lambda_\sigma$ is comparable to $\sigma^{-1}$, we can apply \eqref{lower-MZ} to the second term on the right hand side of \eqref{interpolation-difference-rep} to obtain
\begin{align}\label{lower-MZ-consequence}
\| \L_\sigma(F_\tau -  f)w\|_p^p \le  C_p \sum_{\lambda\in \Lambda_\sigma} (\lambda_+-\lambda)\left|\big(F_\tau(\lambda) - f(\lambda)\big)w(\lambda)\right|^p,
\end{align}
where $\lambda_+ \in \Lambda_\sigma$ is the node rightmost to $\lambda$. The right hand side is a Riemann sum for $|(F_\tau  -f)w|^p$ and taking $\sigma\to\infty$ we obtain
$$
\limsup_{\sigma\to\infty} \| (f - \L_\sigma f)w\|_p \leq (1+{C}^{1/p}_p)\|(f-F_\tau)w\|_p.
$$
The right hand side can now be made arbitrarily small by letting $\tau\to\infty$. This is the ideal situation, but in practice \eqref{lower-MZ} only holds for smoothed versions of  $w$ where the smoothing depends on $\tau$.


The above strategy was initially developed for convergence of Lagrange interpolating polynomials in $L^2[-1,1]$, cf.\ Zygmund \cite[vol.\ II, ch.\ X.7]{Z}. Weighted means for interpolations at zeros of orthogonal polynomials were investigated  by Erd\"os and Turan \cite{ET1}. For a  sample of follow up work we refer to results of Lubinsky, Nevai, Mat\'e, Xu and others focused on polynomial inequalities for Jacobi measures on $[-1,1]$ (see \cite{Lub3, LMN, Nev,Xu1, Xu2} and the references therein). Doubling measures were considered by Mastroianni and Totik \cite{MT} and Mastroianni and Russo \cite{MR}. Many additional references  may be found in the surveys by Lubinsky \cite{Lub5} and \cite[Section 12]{Lub4}. Section 13 of the latter survey contains an overview of polynomial Hermite interpolation.

There is a substantial literature on sampling and interpolation in a given Hilbert space of entire functions and its $L^p$ versions, cf.\ Lyubarski and Seip \cite{LS} and Seip \cite{Seip}. This rests on a deep theory of spaces $\mc{B}^2(\tau,w)$ with bounded evaluation functionals developed by de Branges \cite{dB}. An important role is played by entire functions $E_\tau$ with the property (cf.\ \cite[Theorem 22]{dB})
\[
\int_\R |F(x)  w(x)|^2 dx = \sum_{t \in \Lambda_\tau} \frac{1}{\varphi_\tau'(t)} \left|\frac{F(t)}{E_\tau(t)}\right|^2,
\]
where $\varphi_\tau$ is essentially the argument of $E_\tau$ on the real line.  If $\varphi_\tau'(t)$ is comparable to $\tau$ (with implied constants independent of $t$ and $\tau$), then this identity gives a version of \eqref{lower-MZ} for $p=2$. We show in Section \ref{MZ-Section} that a similar inequality holds if $1<p<\infty$.


In contrast, the question of convergence of \eqref{convergence-condition-one} for $f\in \mc{R}_p(w)$  with $1<p<\infty$   does not appear to have attracted much attention. Inequality \eqref{lower-MZ} without the weight $w$ is due to  P\'olya and Plancherel \cite{Plan1, Plan2, Boas} and implies mean convergence for Lebesgue measure (see also Rahman and Vertesi \cite{RV}). The first weighted result appears to be due to Grozev and Rahman \cite{GR} who considered convergence  with respect to the power weights $ |x|^{a}$ for $a>-1/p$. Their approach  is tailored to power weights and relies on special properties of Bessel functions.



\subsection{Notation and Main Results}

Let $1<p<\infty$. Some of the following statements hold also for $p=\infty$, others for $p=1$, but since our main theorems do not apply to these exponents, we do not consider them here. We denote by $H^p(\C^+)$ the Hardy space of analytic functions $F$ in the upper half plane $\C^+$ for which
\[
\sup_{y>0} \int_\R |F(x+iy)|^p dx<\infty,
\]
and by $H^p(\C^-)$ the Hardy space of the lower half plane.   An entire function $E$ satisfying
\begin{align}\label{HB-ineq}
|E(z)|> |E(\bar{z})|
\end{align}
 for all $\Im z>0$ will be called a Hermite-Biehler function.  Throughout this paper we assume that $E$ has no real zeros.  The $L^p$ de Branges space is defined as follows\footnote{We follow Baranov \cite{Bar1} for the definition of $L^p$ de Branges spaces.}
 \begin{align}\label{intro-HpE}
 \mc{H}^p(E) = \{ F\text{ entire }: F/E, F^*/E\in H^p(\C^+)\},
 \end{align}
 where $F^*$ is the entire function $F^*(z) = \overline{F(\bar{z})}$. If $E$ is of bounded type in $\C^+$, that is, $E$ can be written as a quotient of bounded analytic functions in $\C^+$ then by a result of Krein \cite[Theorem 6.17]{RR}, $E$ is of exponential type $$\delta_E=\limsup_{|z|\to\infty} |z|^{-1}\log |E(z)|\geq 0$$ and it can be shown that $\mc{H}^p(E)$ coincides with the space of entire functions $F$ of exponential type $\leq \delta_E$ such that $F/E\in L^p(\R)$. 

We write $E(z)=A(z)-iB(z)$, where $A$ and $B$ are entire functions that are real-valued for real $z$. We denote by $\mc{Z}_B$ the set of real zeros of $ B$ ($B$ has only real simple zeros by \eqref{HB-ineq} and we will use them as our interpolation nodes). The phase $\varphi$ is defined by the condition $e^{i\varphi(x)} E(x) \in\R$ for all real $x$. The assumption that $E$ has no real zeros implies that $\varphi$ can be chosen so that it has an analytic continuation to an open set containing the real line and is strictly increasing on $\R$ (see Section \ref{modelspace-Section}).
 
Define the formal Lagrange interpolation series (also known as Shannon interpolation in this context)
\begin{align}\label{Ltau-def}
\L_{E} f(z) = \sum_{t\in \mc{Z}_{B}} f(t) \frac{B(z)}{B'(t) (z-t)}.
\end{align}

Throughout the rest of the paper we use the convention $L\lesssim R$ to mean $L\leq c R$ for some constant $c>0$ independent of the objects appearing in $L$ or $R$.  We will sometimes use $\lesssim_{\lambda_1,...,\lambda_n}$ to stress that $c$ may exceptionally depend on objects $\lambda_1,...,\lambda_n$. Finally, we use $L\approx_{\lambda_1,...,\lambda_n}R$ to mean simultaneously $L\lesssim_{\lambda_1,...,\lambda_n}R$ and $R\lesssim_{\lambda_1,...,\lambda_n}L$.

\begin{definition}[$\L$-admissible weight]\label{Lagrangeadmissibleweight}
Let $w:\R\to (0,\infty)$ be  locally bounded and continuous almost everywhere. We say that $w$ is $\L$-admissible  if there exists a family $\{E_\tau:\tau>0\}$ of Hermite-Biehler functions with phase $\varphi_\tau$ such that:
\begin{enumerate}
\item $E_\tau$ has no real zeros and is of bounded type in $\C^+$, with exponential type $\tau$;
\item $|E_\tau(x)|^{-1} \approx w(x)$ for all $x\in\R$;
\item $ \varphi'_\tau(x) \approx \tau$;
\item $B_\tau\notin \mc{H}^p(E_\tau)$ for all $1<p<\infty$ and $\tau>0$.
\end{enumerate}
\end{definition}

 These conditions imply that $\mc{B}^p(\tau,w)=\mc{H}^p(E_\tau)$ with equivalent norms, and in particular $\mc{H}^p(E_\tau) \subseteq \mc{H}^p(E_\sigma)$ if $\sigma\ge \tau>0$. The last condition $B_\tau\notin \mc{H}^p(E_\tau)$ is not a serious restriction, since for every $\alpha\in\R$ the function $e^{i\alpha} E_\tau(z)$ may take the role of $E_\tau$. There can at most be one $\alpha$ for which the entire extension of $-\Im (e^{i\alpha} E_\tau)$ is in the space, since otherwise $E_\tau$ would be an element of the space. By way of a first example we note that $w \equiv 1$ is $\L$-admissible with $E_\tau(z) = e^{-i\tau z}$. 

Throughout this paper we will write $\L_\tau=\L_{E_\tau}$, $\mc{Z}_\tau=\mc{Z}_{B_\tau}$ when convenient.

\begin{theorem}\label{intro-thm1} Let $w$ be an $\L$-admissible weight with $\{E_\tau:\tau>0\}$ as in Definition \ref{Lagrangeadmissibleweight}. If  $f\in \mc{R}_p(w)$, then $\L_\tau f$ defines an entire function in $\mc{B}^p(\tau,w)$ with 
\[
\L_\tau f(\lambda) = f(\lambda)
\]
 for all $\lambda\in  \mc{Z}_{\tau}$. If in addition $\cup_{\tau>0} \mc{B}^p(\tau,w)$ is dense in $L^p(w)$, then
\[
\lim_{\tau\to \infty} \| (f-\L_\tau f)w\|_p=0.
\]
\end{theorem}

For conditions when che class of  functions of exponential type is dense in $L^p(w)$ we refer to  Koosis \cite[Ch.\ VI]{Koo}. (Our examples deal with $w$ of polynomial growth where an elementary construction gives the necessary density statement.)

In order to establish convergence of Hermite interpolations, few modifications need to be made, the most important is that interpolations now live in $\mc{B}^p(2\tau,w^2)$. The reproducing kernel associated with $E$ is
\begin{align}\label{Kdef}
 K_E(w,z) = \frac{E(z) E^*(\bar{w}) - E^*(z) E(\bar{w})}{2\pi i(\bar{w} -z)}=\frac{B(z)A(\bar{w}) - A(z)B(\bar{w})}{\pi(z-\bar{w})}.
\end{align}
Setting for $t\in\mc{Z}_B$ and $z\in \C$ 
\begin{align*}
U_E(t,z) &=  \frac{K(t,z)^2}{K(t,t)^2}\left(1-2\frac{K'(t,t)(z-t)}{K(t,t)}\right)=\frac{B(z)^2}{B'(t)^2(z-t)^2} - \frac{ B(z)^2B''(t)}{B'(t)^3(z-t)} \\
V_E(t,z) &= \frac{K(t,z)^2(z-t)}{K(t,t)^2}=\frac{B(z)^2}{B'(t)^2(z-t)},
\end{align*}
we define the formal Hermite interpolation series by
\begin{align}\label{Htau-def}
 \H_E(f,z) = \sum_{t\in \mc{Z}_B} f(t) U_E(t,z) +   \sum_{t\in \mc{Z}_B} f'(t) V_E(t,z).
\end{align}

\begin{definition}[$\H$-admissible weight]\label{Hermiteadmissibleweight}
We say that a weight $w$ is $\H$-admissible if $w$ it is $\L$-admissible and the associated Hermite-Biehler functions $E_\tau$ have the following additional property
\begin{equation}\label{E'divEbound}
\|E'_\tau/E_\tau\|_{H^\infty}\lesssim \tau.
\end{equation}
\end{definition}
These conditions again imply that $\mc{B}^p(2\tau,w^2)=\mc{H}^p(E^2_\tau)$, and in particular $\mc{H}^p(E^2_\tau) \subseteq \mc{H}^p(E^2_\sigma)$ if $\sigma\ge \tau>0$. We remark that by \cite[Theorem 1 \& Corollary ~7]{GL}, the space $\mc{H}^p(E_\tau^2)$ is closed under differentiation and $B_\tau\notin \mc{H}^2(E_\tau^2)$, so condition $(d)$ in the definition of $\L$-admissibility is actually implied by \eqref{E'divEbound} and could be removed. In the following, we write $\H_\tau$ in place of $\H_{E_\tau}$. 

\begin{theorem}\label{intro-thm2} Let $w$ be an $\H$-admissible weight. If $f\in\mc{R}^{(1)}(w^2)$, then $\H_\tau f$ defines an entire function in $\mc{B}^p(2\tau,w^2)$ with
\begin{align*}
\H_\tau (f,\lambda) = f(\lambda) \ \text{ and } \ 
\H_\tau'(f,\lambda) = f'(\lambda)
\end{align*}
for all $\lambda\in \mc{Z}_B$.  If in addition $\cup_{\tau>0} \mc{B}^p(2\tau,w^2)$ is dense in $L^p(w^2)$ then
\[
\lim_{\tau\to \infty} \|(f-\H_\tau f)w^2\|_p =0.
\]
\end{theorem}
 We remark that $\|(f'-F_\tau')w^2\|_p$ is not required to converge to zero as $\tau\to \infty$ (and in most cases it will not converge).
 
\subsection{Power Weights}
We show next that the result of Grozev and Rahman \cite{GR} is a special case of our convergence results. For $\nu>-1$ define entire functions
\begin{align*}
A_{\nu}(z) &= \Gamma(\nu+1) (z/2)^{-\nu} J_\nu(z) \\
B_{\nu}(z) &= \Gamma(\nu+1) (z/2)^{-\nu} J_{\nu+1}(z),
\end{align*}
where $J_\nu$ is the Bessel function of order $\nu$ of the first kind. Define $E_{\nu} = A_{\nu} -iB_{\nu}$. For $\alpha\in [0,\pi)$ we set  $E_{\nu,\tau,\alpha}(z) = e^{i\alpha}\tau^{\nu+\frac12} E_\nu(\tau z)$ and define real entire $A_{\nu,\tau,\alpha}$, $B_{\nu,\tau,\alpha}$ by $E_{\nu,\tau,\alpha} = A_{\nu,\tau,\alpha} - iB_{\nu,\tau,\alpha}$.  Clearly $\mc{H}^p(E_{\nu,\tau,\alpha})=\mc{H}^p(E_{\nu,\tau,0})$, but the nodes of interpolation will be different.  We collect required material about these functions  in the following lemma (see de Branges  \cite[section 50]{dB} and \cite[Section 4.1]{GL}).  

 \begin{lemma}\label{Enu-properties}  The following properties hold.
 \begin{enumerate}
\item\label{item7} $E_\nu$ is a Hermite-Biehler function of exponential type $1$.

 \item\label{item1}
$
|E_{\nu}(x)|^{-1} \approx_\nu \max(1,|x|)^{\nu+\frac12}
$

\item\label{item3} $\varphi'_{\nu}(x) \approx_\nu 1$

\item\label{item4} The Bessel zeros $t_{\nu,k}$ of $B_\nu$ satisfy
$
t_{\nu,k+1}- t_{\nu,k} = \pi k + \mc{O}_\nu(k^{-1}).
$
\item\label{item6} $B_\nu\notin \mc{H}^p(E_\nu)$. 

\item Assume $\nu\ge -\frac12$ and $1<p<\infty$, or $-\frac12>\nu>-1$ and $1<p<|\nu+\frac12|^{-1}$.  An entire function $F$  of exponential type $\le 1$ satisfies
\[
\int_\R |F(x)x^{\nu+\frac12}|^p dx <\infty
\]
if, and only if, $F\in \mc{H}^p(E_\nu)$. 

\item For $\alpha\in\R$ the union $\cup_{\tau>0} \mc{H}^p(E_{\nu,\tau,\alpha})$ is dense in $\mc{R}_p(|x|^{\nu+\frac12})$ under the same conditions as item {\rm (f)}.
\end{enumerate}
 \end{lemma}
 \begin{proof}
 We only sketch the proof and leave details to the reader. Item (a) can be found in \cite[section 50]{dB}. Items (b) and (e) are a consequence of the classical asymptotic expansion of $J_\nu(x)$ for large $x$. The same asymptotic can be used in conjunction with the differential equations defining the Bessel function $J_\nu$ to show $\varphi'_\nu(x)=\Re [i E'_\nu(x)/E_\nu(x)] =1-(2\nu+1)\frac{A_\nu(x)B_\nu(x)}{|E_\nu(x)|^2}$, which proves items (c) and (d). Item (f) follows from items (a) and (b) (with some work). Item (g) is classical and can be done using convolutions, approximations of the Dirac delta and further tricks to deal with the singularity at the origin.
 \end{proof}

 We let $\L_{\nu,\tau,\alpha}=\L_{E_{\nu,\tau,\alpha}}$ be the Lagrange interpolation operator with nodes 
 \begin{align*}
 \mc{Z}_{\nu,\tau,\alpha} & =\{t\in\R: B_{\nu,\tau,\alpha}(t)=0\} \\
 & = \{t \in \R\setminus\{ 0\}:J_\nu(\tau t)/ J_{\nu+1}(\tau t) =-\tan(\alpha)\}  \ \ (\text{if } \alpha\neq \pi/2 )
 \\ & =\{t \in \R : t J_\nu(\tau t)=0 \}   \ \ (\text{if } \alpha= \pi/2 )
  \end{align*}
 The choice $\alpha=\frac{\pi}2$ in the following corollary recovers the results of \cite{GR},  the proof is provided in Section \ref{proof-Section}.

\begin{corollary}\label{homog-theorem}  Assume $\nu\ge -\frac12$ and $1<p<\infty$, or $-\frac12>\nu>-1$ and $1<p<|\nu+\frac12|^{-1}$.   If  $f\in \mc{R}_p(|x|^{\nu+\frac12})$ then $|x|^{\nu+1/2}\L_{\nu,\tau,\alpha} f (x)\in L^p(\R)$, $\L_{\nu,\tau,\alpha} f$ has exponential type at most one and
\[
\lim_{\tau\to \infty} \int_\R \left|(f(x) -\L_{E_{\nu,\tau,\alpha}} f(x)) |x|^{\nu+\frac12}\right|^pdx =0.
\]
\end{corollary}

\section{Background}\label{modelspace-Section}

\subsection{De Branges Spaces} This section collects known facts about the $L^p$ de Branges space $\mc{H}^p(E)$ (cf.\ de Branges \cite[pp. 50 - 59]{dB} for $p=2$ and Baranov \cite{Bar1} for $p\neq 2$).  Recalling \eqref{Kdef}, it follows that $x\mapsto K(w,x)/E(x)\in L^q(\R)$ for every $1<q\le \infty$, and the representation
\begin{align}\label{dBreproduction}
F(w) = \int_\R \frac{F(x) \overline{K(w,x)}}{|E(x)|^2} dx,
\end{align}
follows from Cauchy's formula for all $F\in \mc{H}^p(E)$.  In particular, the space $\mc{H}^2(E)$ is a Hilbert space with reproducing kernel $K(w,z)$.  A direct calculation shows that
\begin{align}\label{phi-id}
\varphi'(x)=\Re \bigg\{i\frac{E'(x)}{E(x)}\bigg\}=\pi\frac{K(x,x)}{|E(x)|^2} > 0
\end{align}
for all real $x$. From $e^{2i\varphi(x)} E(x)^2 = |E(x)|^2$ we obtain
$$
e^{-2i\varphi(x)}   = \frac{A(x)^2}{|E(x)|^2}  - \frac{B(x)^2}{|E(x)|^2} + 2i\frac{A(x)B(x)}{|E(x)|^2}
$$
for all real $x$, and as a consequence, if $E$ has no real zeros then $\mc{Z}_{B}=\varphi^{-1}(\pi \Z)$.

In the usual abuse of notation we identify $H^p(\C^+)$ as the subspace of $L^p(\R)$ consisting of non-tangential boundary values of elements in $H^p(\C^+)$. We also denote by $Hf$ the Hilbert transform of $f$, and recall that for $1<p<\infty$ the   Riesz projection $P_+$, given for $f\in L^p(\R)$ by
\[
P_+ f = \frac12( f+iHf),
\]
defines a bounded operator from $L^p(\R)$ onto $H^p(\C^+)$.

\subsection{Connection with Model Spaces} The standard source for model spaces is the book of Nikolski \cite[Chapter 6]{Ni}, but also   \cite{Bar2, Bar1,CP,GMR}. Recall that an inner function for $\C^+$ is bounded by $1$ in the upper half plane and its modulus has boundary value equal to $1$ almost everywhere on the real line. By \eqref{HB-ineq} the function meromorphic function $\Theta = E^*/E$ is inner for $\C^+$  and satisfies $\Theta^*=1/\Theta$. The model space $\mc{K}_\Theta^p$ is defined as the kernel
\[
\mc{K}_\Theta^p = \ker T_{\Theta^*}
\]
where the Toeplitz operator $T_{\Theta^*}:H^p(\C^+)\to H^p(\C^+)$  is given by 
\[
T_{\Theta^*} f= P_+(\Theta^* f).
\]
The map $F\mapsto F/E$ defines an isometry between $\mc{H}^p(E)$ and $\mc{K}_\Theta^p$  (this is a consequence of the equivalence $F^*/E\in H^p(\C^+)$ if and only if $F/E^* \in H^p(\C^-)$, cf.\ Baranov \cite[Theorem 2.1]{Bar2}). The space $\mc{K}_\Theta^2$ is a reproducing kernel space with kernel $k$ given by
\begin{align}\label{rep-kernel}
k(w,z) = \frac{K(w,z)}{E(z)E^*(\bar{w})} = \frac{i}{2\pi} \frac{1-\Theta(z) \Theta^*(\bar{w})}{z-\bar{w}}.
\end{align}

 We define the integral operator
\begin{align}\label{SEf-def}
S_{E}f(z) =  \int_\R f(u) K(u,z) \frac{du}{|E(u)|^2}
\end{align}
and note that by \eqref{rep-kernel} this  is the difference of the Hilbert transform of $f/E$ and of the multiplication by $\Theta$ of the Hilbert transform of $\Theta^* f/E$. For easy reference we collect the boundedness of $f/E\mapsto S_Ef/E$ on $L^p$ in the following lemma (cf.\ Hollenbeck and Verbitsky \cite{HV}  for the constant). 

\begin{lemma}\label{SEf-lemma}  Let $1<p<\infty$. If $f/E\in L^p(\R)$, then $S_Ef\in \mc{H}^p(E)$, and 
\begin{align}\label{SEbound}
\int_\R \left|\frac{S_E f(x)}{E(x)}\right|^p dx  \le 2  \left(\csc \tfrac\pi p\right)^{2p} \int_\R \left|\frac{f(x)}{E(x)}\right|^p dx.
\end{align}
\end{lemma}

We remark that $S_Ef$ is related to the projection operator $P_\Theta:L^p(\R) \to \mc{K}_\Theta^p$ given by
\[
P_\Theta f = P_+(f) - \Theta P_+(\Theta^* P_+(f))
\] 
through the identity $ P_\Theta(f/E) = (1/E) S_E f$, but we do not require this connection.

In order to analyze convergence of interpolation with derivatives we require the following two operators, whose form is suggested by the interpolation kernels $U_E$ and $V_E$ in the definition of $\H_E f$. For $t\in\R$ and $g\in L^p(\R)$ we define
\begin{align*}
D_Eg(t) &= \int_\R g(x) k(t,x)^2  \left(1-2\frac{K'(t,t)(x-t)}{K(t,t)}\right) dx,\\
T_Eg(t) &= \int_\R g(x) (x-t) k(t,x)^2  dx
\end{align*}
where the prime denotes differentiation in the second variable. As before, we write $D_\tau$ and $T_\tau$ if $E = E_\tau$.

\begin{lemma}\label{projection-bounds} Let $1<p<\infty$. Assume $E_\tau$ be a Hermite-Biehler function with $\varphi_\tau'(x) \approx \tau$ and $\|E_\tau'/E_\tau\|_{H^\infty} \lesssim \tau$, for $\tau>0$. If $g\in L^p(\R)$, then
\begin{align*}
\|D_\tau g\|_p&\lesssim_p \tau \|g\|_p,\\
\|T_\tau g\|_p &\lesssim_p \|g\|_p.
\end{align*}
Furthermore, $D_\tau g, T_\tau g\in \mc{K}_{(E_\tau^*)^2/E_\tau^2}^p$. 
\end{lemma}

\begin{proof} We omit the subscript $\tau$ for the functions throughout this proof. We observe that $z\mapsto k(t,z)^2$ and $z\mapsto (z-t)k(t,z)^2$ are in $L^q(\R)$ where $q$ is the conjugate exponent of $p$. Hence the integrals defining $D_\tau g$ and $T_\tau g$ are absolutely convergent and define entire functions. A direct calculation gives for $t\in\R$
\[
T_\tau g(t) = \int_\R \frac{1}{x-t} \left( \frac{1 - \frac{E^*(x)}{E(x)} \frac{E(t)}{E^*(t)} }{2\pi i}\right)^2 g(x) dx,
\]
and after expanding the square, $T_\tau g$ is a sum of three terms, each of which involves multiplications with functions of constant modulus (not depending on $\tau$) and a Hilbert transform.  Hence $T_\tau$ is a bounded operator on $L^p(\R)$ with norm depending only on $p$. 

Since $E'/E$ is bounded in $\C^+$, it follows from \cite{Bar1} that differentiation defines a bounded operator on $\mc{H}^p(E_\tau)$ with norm $\lesssim_p \|E_\tau'/E_\tau\|_{H^\infty}$. Identity \eqref{dBreproduction} applied to $z\mapsto \frac{\partial}{\partial z}K'(t,z)$ gives
\[
\frac{\partial K(t,z)}{\partial z} = \int_\R \left( \frac{d}{dx}K(t,x)\right) \overline{K(z,x)} \frac{dx}{|E(x)|^2}.
\]
We apply Cauchy-Schwarz and use the norm of the differentiation operator to get
\[
|K'(t,z)|\lesssim \tau K(t,t)^{1/2} K(z,z)^{1/2}.
\]
Letting $z=t$ leads to 
$
\left| \frac{K'(t,t)}{K(t,t)}\right|\lesssim \tau.
$
To estimate $D_\tau$, we obtain from \eqref{dBreproduction} that
\[
\int_\R |k(t,x)|^2 dx = |E(t)|^{-2} K(t,t) = \pi \varphi'(t) \lesssim \tau.
\]
Applying the integral operator version of Young's inequality (see \cite[Theorem 0.3.1]{Sogge}) gives
\[
\left(\int_\R \left| \int_\R k(t,x)^2 g(x) dx\right|^p dt \right)^{1/p} \lesssim  \tau \|g\|_p,
\]
which finishes he proof of the claimed inequality.

For the final statement we note first that $z\mapsto K(w,z)^2/E(z)^4$ is an element of $H^p(\C^+)$, and the same is true for $K^*(w,z)^2/E(z)^4$. Hence $K(w,z)^2$ is an element of $\mc{H}^p(E^2)$, and it follows that $k(t,z)^2$ is in $\mc{K}_{(E^*)^2/E^2}^p$. The proof for $(x-t)k(t,z)^2$ is analogous. It follows that these functions are in the kernel of the Toeplitz operator for this model space, and integrating in $t$ while observing the norm inequalities proved above shows that $D_\tau g$ and $T_\tau g$ are in the same model space.
\end{proof}


\section{Marcinkiewicz Inequalities}\label{MZ-Section}

Throughout this section $1<p<\infty$ and $\{E_\tau:\tau>0\}$ is a family of Hermite-Biehler functions with no real zeros and phase $\varphi_\tau$ such that 
\begin{align}\label{phi-bound}
\varphi_\tau'(x)\approx \tau.
\end{align}
As mentioned in the introduction, in order to obtain a version of \eqref{lower-MZ} we  start with an upper Marcinkiewicz inequality in $\mc{H}(E_\tau)$, prove a convergence statement for interpolations (which we require for \eqref{convergence-condition-one} as well), and combine those ingredients to obtain a lower Marcinkiewicz inequality.

Since $E_\tau^*/E_\tau = e^{2i\varphi_\tau}$ on $\R$, a result of Dyakonov \cite{Dy94} and \cite[eq.\ (3.1)]{Dy} implies that
\begin{align}\label{derivative-bound}
\| (F/E_\tau)' \|_p \le C_p\tau \|F/E_\tau\|_p.
\end{align}
Furthermore, \eqref{phi-bound} gives $\pi=|\varphi_\tau(t_+) - \varphi_\tau(t)| \approx \tau|t_+-t|$, if $t_+>t$ are consecutive zeros of $B_\tau$. Hence
\begin{align}\label{zero-spacing}
|t_+- t| \approx \tau^{-1}.
\end{align}
We review next an upper Marcinkiewicz inequality, due to Baranov.

\begin{lemma}[cf.\ {\cite[Theorem 5.1]{Bar2}}]\label{baranov-lemma} Let $E_\tau$ satisfy \eqref{phi-bound}. Then for all $F\in \mc{H}^p(E_\tau)$
\[
\frac1\tau \sum_{\lambda\in  \mc{Z}_{\tau}} \left|\frac{F(\lambda)}{E_\tau(\lambda)} \right|^p \lesssim_p  \|F/E_\tau\|_p^p.
\] 
\end{lemma}

\begin{proof} This may be obtained by considering the sum on the left as integration against a Carleson measure with masses at the points of $ \mc{Z}_{\tau}$ and observing that the proof of \cite[Theorem 5.1]{Bar2}  carries over from $p=2$.

Alternatively, following \cite[Theorem 1]{MR}, set $\lambda_+ = \inf\{t \in \R: B_\tau(t)=0\text{ and }t>\lambda\}$. Starting point is the inequality
  \[
  |h(x)|^p(y-x) \le 2^{p-1} \left(\int_x^y |h(u)|^p du + (y-x)^p \int_x^y|h'(u)|^p du\right),
  \]
  valid for all $h\in C^1$ and $x<y$.  We use consecutive zeros of $B_\tau$ for the endpoints and apply this with the $C^\infty(\R)$-function $h=F/E_\tau$. This leads to
\[
  \sum_{\lambda\in  \mc{Z}_{\tau}} \left|\frac{F(\lambda)}{E_\tau(\lambda)}\right|^p \tau^{-1} \le C_p \left(\int_\R \left|\frac{F(x)}{E_\tau(x)}\right|^p dx + \tau^{-p} \int_\R \left| \frac{d}{dx}\left[ \frac{F(x)}{E_\tau(x)}\right]\right|^p dx\right),
\]
 and \eqref{derivative-bound} implies the claim.
\end{proof}

The final statement in this section is a pointwise bound for the series defining $\L_\tau f$ in order to establish when it defines an entire function.

\begin{lemma}\label{uniform-conv} Let $f\in \mc{R}_p(w)$. If \eqref{phi-bound} and conditions (a) and (b) of Definition \ref{Lagrangeadmissibleweight} hold, then the series defining $\L_{\tau}f$ converges uniformly on compact subsets of $\C$, and for $z\in\C$
\[
|\L_\tau f (z)|\le C_p(\|f w\|_p +1) \left(\frac1\tau\sum_{t\in  \mc{Z}_{\tau}} \left|\frac{ B_\tau(z)}{z-t}\right|^q\right)^{1/q}.
\]
\end{lemma}

\begin{proof}  We drop subscripts $\tau$. Let $z$ be in a compact subset $\Gamma$ of $\C\backslash \mc{Z}_{B}$. It follows from  \eqref{phi-id} that $\varphi'(t) = A(t)^{-1}B'(t)$ for $t\in \mc{Z}_B$, hence multiplying and dividing by $A(t)$ gives
\begin{align*}
|\L_\tau f(z)| &= \left|\sum_{t\in \mc{Z}_B} \frac{f(t)B(z)}{B'(t)(z-t)}\right|\\
&\le \left(\sum_{t\in \mc{Z}_{B}}\frac{1}{\varphi'(t)} \left| \frac{f(t)}{A(t)}\right|^p \right)^{1/p} \left(|B(z)|^q\sum_{t\in \mc{Z}_{B}}\frac{1}{\varphi'(t)}\left|\frac{ 1}{z-t}\right|^q\right)^{1/q}.
\end{align*}
Since $E = A$ on $\mc{Z}_{B}$, the first term is comparable to the Riemann sum of $\|f w\|_p$. It follows from \eqref{phi-bound} and \eqref{derivative-bound} that the second series converges uniformly and absolutely for $z\in \Gamma$.  Since the singularities at $\mc{Z}_B$ are removable,  $\L_\tau f$ defines an entire function. 
\end{proof}


\subsection{Lagrange Interpolation} We prove next a version of \eqref{lower-MZ}. We remark that inequalities of this type are known for considerably more general measures, cf.\ Volberg \cite[Theorem 2]{V}, if the constant is not required to depend explicitly on $\tau$.

 \begin{proposition}\label{HpE-rep} Let $E_\tau$ satisfy \eqref{phi-bound} and $B_\tau\notin \mc{H}^p(E_\tau)$. Then for $F\in \mc{H}^p(E_\tau)$
\begin{align}\label{HpE-interpolation}
 F(z) = \sum_{t\in  \mc{Z}_{\tau}} F(t) \frac{K_\tau(t,z)}{K_\tau(t,t)}
 \end{align}
  in $\mc{H}^p(E_\tau)$ and uniformly on compact subsets of $\C$. Moreover,
\begin{align}\label{lower-MZ-inequ}
\|F/E_\tau\|_p^p \lesssim_p \frac{1}{\tau} \sum_{t\in \mc{Z}_B} \left|\frac{F(t)}{E_\tau(t)}\right|^p.
\end{align}
 \end{proposition}

\begin{proof} Let $1<p<\infty$. We drop the subscript $\tau$ throughout this proof.  Let $F\in \mc{H}^p(E)$. Let $F_k$ be the partial sum  of the series in \eqref{HpE-interpolation} using the summands with $|t|\le k$. This is an element of $\mc{H}^p(E)$, and we show first that it forms a Cauchy sequence in this space. It is known\footnote{For model spaces on the unit disk this may be found in \cite[Lemma 4.2]{Co}, and the proof for the upper half plane is analogous.} that $\mc{K}^q_{E^*/E}$ is norm equivalent to the dual space of $\mc{K}_{E^*/E}^p$ (here $p^{-1} + q^{-1} =1$) and one can show $\mc{H}^p(E)'=\mc{H}^q(E)$.  This gives
\begin{align}\label{sup-rep}
\left\|\frac{F_k-F_m}{E}\right\|_p \lesssim_p  \sup \left|\int_\R \frac{F_k(x)-F_m(x)}{E(x)} \frac{\overline{G(x)}}{\overline{E(x)}} dx \right|
\end{align}
where the supremum is taken over all $G \in \mc{H}^q(E)$ with $\|G/E\|_q=1$. It follows from \eqref{dBreproduction} that
\begin{align*}
\int_\R \frac{F_k(x)-F_m(x)}{E(x)} \frac{\overline{G(x)}}{\overline{E(x)}} dx &= \sum_{\substack{t\in \mc{Z}_B \\ k<|t|\le m}} \frac{F(t)}{K(t,t)} \int_\R \frac{K(t,x) \overline{G(x)}}{|E(x)|^2} dx \\
& = \sum_{\substack{t\in \mc{Z}_B \\ k<|t|\le m}} \frac{F(t) \overline{G(t)}}{K(t,t)}.
\end{align*}
By \eqref{phi-id} we have $K(t,t) = \varphi'(t)|E(t)|^2$. Using H\"older's inequality and dropping the restriction on $t$ in the last series of the following inequality gives
\begin{align}\label{uniform-bounds} 
 \left|\sum_{\substack{ t\in \mc{Z}_B \\ k<|t|\le m}} \frac{F(t)\overline{G(t)}}{\varphi'(t)|E(t)|^2} \right| 
  \le \left( \sum_{\substack{ t\in \mc{Z}_B \\ k<|t|\le m}} \frac{1}{\varphi'(t)} \left|\frac{F(t)}{E(t)}\right|^p\right)^{\frac{1}{p}} \left( \sum_{t\in \mc{Z}_B} \frac{1}{\varphi'(t)} \left|\frac{G(t)}{E(t)}\right|^q\right)^{\frac1q}.
\end{align}
The assumption \eqref{phi-bound} and Lemma \ref{baranov-lemma} lead to
 \[
\left( \sum_{t\in \mc{Z}_B} \frac{1}{\varphi'(t)} \left|\frac{G(t)}{E(t)}\right|^q\right)^{1/q} \lesssim_q \|G/E\|_q =1.
\]
Since the constants do not depend on $G$, the inequalities hold for the supremum in \eqref{sup-rep} as well, and it follows that $F_k$ converges in $\mc{H}^p(E)$. Alternatively, without using the dual space representation, the function $G$ in \eqref{sup-rep} may be replaced by any $h$ with $h/E\in L^q(\R)$. This leads to the use of $S_Eh$ in place of $G$ and an application of Lemma \ref{SEf-lemma} in the last step.

We show next that the limit of $F_k$ is $F$. It follows from Lemma \ref{uniform-conv} that the series converges uniformly in compact subsets of $\C$ and hence defines an entire function. To show that it represents $F$ fix $w\notin \R$ with $F(w)\neq 0$ and note that $G_w$ defined by
\[
G_w(z) = \frac{F(z) B(w) - B(z) F(w)}{z-w}
\]
is entire and an element of $\mc{H}^2(E)$. It follows from the theory of de Branges spaces for $p=2$ (cf.\ the proof of \cite[Theorem 22]{dB})  that the representation \eqref{HpE-interpolation} holds for $G_w$, and this may be rewritten as
\[
\frac{F(z)}{B(z)} - \frac{F(w)}{B(w)} = \sum_{t\in \mc{Z}_B} \frac{F(t)}{B'(t)} \left[\frac{1}{z-t}+\frac{1}{t-w}\right].
\]
Hence for some constant $c_F$
\[
F(z) =  \sum_{t\in \mc{Z}_B} \frac{F(t) B(z)}{B'(t)(z-t)} + c_F B(z),
\]
and as we had seen, the series converges in $\mc{H}^p(E)$. Since by assumption $B\notin \mc{H}^p(E)$, we must have $c_F=0$.

Finally, returning to \eqref{sup-rep}, if we replace $F_k-F_m$ by $F$  then the same calculations lead to \eqref{uniform-bounds} without the restriction $k<|t|\le m$. Raising the resulting inequalities to the $p$th power and using \eqref{phi-bound} leads to \eqref{lower-MZ-inequ}.
\end{proof}

\subsection{Hermite Interpolation} The interpolation with derivatives requires the following version of \eqref{lower-MZ}.

\begin{proposition}\label{prop:der-lower-MZineq}
Let $E_\tau$ satisfy \eqref{phi-bound} and $B_\tau^2\notin\mc{H}^p(E_\tau^2)$. Then for $F\in \mc{H}^p(E_\tau^2)$
\[
F(z) =  \sum_{t\in \mc{Z}_\tau} \left[F(t) \frac{K_\tau(t,z)^2}{K_\tau(t,t)^2}\left( 1 - 2 \frac{K_\tau'(t,t)(z-t)}{K_\tau(t,t)}\right) + F'(t) \frac{K_\tau(t,z)^2(z-t)}{K_\tau(t,t)^2}\right]
\]
in $\mc{H}^p(E_\tau^2)$ and uniformly on compact subsets of $\C$, and
\[
\|F/E_\tau^2\|_p  \lesssim_p  \left( \frac{1}\tau \sum_{t\in\mc{Z}_\tau}\left|\frac{F(t)}{E_\tau(t)^2}\right|^p\right)^{1/p} + \frac{1}\tau  \left( \frac{1}\tau \sum_{t\in\mc{Z}_\tau}\left|\frac{F'(t)}{E_\tau(t)^2}\right|^p\right)^{1/p}.
\]
\end{proposition}

\begin{proof}
 Following the strategy of the previous section, we express $\|F/E_\tau^2\|_p$ using duality, plug in the partial sums of the proposed interpolating series $\H_{\tau} f$, change summation, and obtain sums of certain integral transforms that are bounded using Lemma \ref{projection-bounds}.  The proof follows the same lines with only few modifications and we leave the details to reader, but we mention that the necessary local convergence result of the interpolating series was proved in \cite{Gon}. 
 \end{proof}

\section{Proofs of the Main Results}\label{proof-Section}

\begin{proof}[\bf Proof of Theorem \ref{intro-thm1}]  Let $f\in \mc{R}_p(w)$. It follows from Lemma \ref{uniform-conv} that $\L_\tau f$ defines an entire function. The partial sums $L_k$ of $\L_\tau f$ are in $\mc{H}^p(E_\tau)=\mc{B}^p(\tau,w)$. Since $L_k(t) = f(t)$ for  $|t|\le k$ and $L_k(t) =0$ otherwise for $t\in  \mc{Z}_{\tau}$, we have
$$
\|(L_k-L_n)/E_\tau\|_p^p \lesssim_p \frac{1}{\tau} \sum_{\substack{t\in  \mc{Z}_{\tau} \\ k<|t|\le n}} \left| f(t) w(t)\right|^p.
$$
The right hand side is a partial sum of a  convergent Riemann sum, so $\L_\tau f$ defines an element of $\mc{H}^p(E_\tau)$, which equals $\mc{B}^p(\tau,w)$. We can then apply Proposition \ref{HpE-rep} to obtain the desired lower Marcinkiewicz-Zygmund inequality \eqref{lower-MZ} for $\mc{B}^p(\tau,w)$ and $\Lambda_\tau=\mc{Z}_\tau$, and the remaining part of the proof follows from the argument presented thereafter.
\end{proof}

\begin{proof}[\bf Proof of Theorem \ref{intro-thm2}] 
 Let $f\in \mc{R}^1_p(w)$. The proof that each series in the definition of $\H_\tau f$ converges uniformly on compact subsets of $\C$ and hence defines an entire function is a calculation analogous to Lemma \ref{uniform-conv} (note that $K_\tau'(t,t)/K_\tau(t,t)\le C\tau$ by the proof of Lemma \ref{projection-bounds}). The proof can be finished by a very similar argument to the proof of Theorem \ref{intro-thm2}, but now we need to apply Proposition \ref{prop:der-lower-MZineq} and the fact that $\H^p(E_\tau^2)$ is closed under differentiation, we leave the details to the reader.
\end{proof}

\begin{proof}[\bf Proof of Corollary \ref{homog-theorem}]
In what follows we will use the results of Lemma \ref{Enu-properties} and we omit $\alpha$. Let $w_\tau(x)=\max(\tau^{-1},|x|)^{\nu+1/2}$ for $\tau>0$ and $w_\infty(x)=|x|^{\nu+1/2}$. Consider the family of Hermite-Biehler functions $\{E_{\nu,\tau}\}_{\tau>0}$. Since $|E_{\nu,\tau}^{-1}|\approx w_\tau$,  this family satisfies all conditions of Definition \ref{Lagrangeadmissibleweight} if we replace $w$ by $w_\tau$ and the proof of Theorem \ref{intro-thm1} can be replicated line by line to show $\|(f-\L_{\nu,\tau} f)w_\tau\|_p\to 0$ as $\tau\to\infty$ and $\L_{\nu,\tau} f  \in \mc{H}(E_{\nu,\tau})$. Hence $\L_{\nu,\tau} f$ has exponential type at most one and $w_\infty \L_{\nu,\tau} f\in L^p(\R)$. 
Since $w_\tau\geq w_\infty$ for $\nu\geq -1/2$ we obtain
\begin{align*}
\|(f-\L_{\nu,\tau} f)(w_\tau-w_\infty)\|_p \leq \|(f-\L_{\nu,\tau} f)w_\tau\|_p \to 0,
\end{align*}
hence $\|(f-\L_{\nu,\tau} f)w_\infty\|_p^p\to 0$, which finishes the proof in the case $\nu\geq -1/2$.

For $\nu<-\frac12$ we have instead
\begin{align*}
\|(f-\L_{\nu,\tau} f)(w_\tau-w_\infty)\|_p  \leq \|fw_\infty\|_{L^p([-\tau^{-1},\tau^{-1}])} + \|\L_{\nu,\tau} fw_\infty\|_{L^p([-\tau^{-1},\tau^{-1}])}.
\end{align*}
The integral of $|f w_\infty|^p$ restricted to $[-\tau^{-1},\tau^{-1}]$ converges to zero so we need to analyze the contribution from $\L_{\nu,\tau} f w_\infty$. A scaling argument in the inequality of Lemma \ref{uniform-conv} may be used to show $|\L_{\nu,\tau}f(x)| \lesssim \tau^{1-\frac1q -|\nu+\frac12|},$ and hence 
\begin{align}\label{LSizeAtZero}
|\L_{\nu,\tau} f(x)/E_{\nu,\tau}(x)|^p \lesssim \tau
\end{align}
for $|x|\le \tau^{-1}$. Since the integral of $|f w_\infty|^p$ and hence of $|f /E_{\nu,\tau}|^p$ converges to zero, we also have the limit relation $\|\L_{\nu,\tau} f/E_{\nu,\tau}\|_{L^p([-\tau^{-1},\tau^{-1}])}\to 0$ as $\tau\to \infty$. Defining for $\varepsilon>0$
\[
X_{\tau,\varepsilon} = \{|x|\le \tau^{-1}: \varepsilon\tau \le |\L_{\nu,\tau}f(x)/E_{\nu,\tau}(x)|^p\},
\]
we split the integral of $|\L_{\nu,\tau} f w_\infty|^p$ on $[-\tau^{-1},\tau^{-1}]$ into the integral over $X_{\tau,\varepsilon}$ and its complement, utilize on the complement  the estimate $$|\L_{\nu,\tau} f(x) w_\infty(x)|^p < \varepsilon \tau |E_{\nu,\tau}(x) w_\infty(x)|^p,$$ combine this with $|E_{\nu,\tau}(x)|^p \lesssim \tau^{-p|\nu+\frac12|}$, use \eqref{LSizeAtZero} on $X_{\tau,\varepsilon}$, and observe that due to the shape of $w_\infty$ the contribution from $X_{\tau,\varepsilon}$ is largest if this set is an interval with center at the origin (we leave the details to the reader). This finishes the proof. 
\end{proof}



 \end{document}